\documentclass[11pt]{article}
\usepackage[margin=1in]{geometry}
\usepackage{graphicx}
\usepackage{multirow}
\usepackage{amsmath,amssymb,amsfonts}
\usepackage{amsthm}
\usepackage{mathrsfs}
\usepackage{xcolor}
\usepackage{textcomp}
\usepackage{booktabs}
\usepackage{makecell}
\usepackage{authblk}

\theoremstyle{plain}
\newtheorem{theorem}{Theorem}[section]

\newtheorem{lemma}{Lemma}[section]
\newtheorem{corollary}[theorem]{Corollary}
\theoremstyle{definition}
\newtheorem{Definition}{Definition}[section]
\theoremstyle{remark}

\newtheorem{remark}{Remark}[section]
\numberwithin{equation}{section}

\begin{document}

\title{On the sum of odd minimal excludants over overpartitions}

\author[1]{Veena V S\thanks{Corresponding author. Email: veenavsmath@gmail.com}}
\author[2]{S N Fathima\thanks{Email: dr.fathima.sn@gmail.com}}
\affil[1]{Department of Data Analytics and Mathematical Sciences, Jain (Deemed-to-be University), Kochi, Kerala, India}
\affil[2]{Department of Mathematics, Pondicherry University, Puducherry, India}
\date{}

\maketitle

\begin{abstract}
Andrews and Newman introduced the minimal excludant $\mathrm{mex}(\lambda)$
of an integer partition $\lambda$ and studied the summatory function
$\sigma\mathrm{mex}(n)$, and Baruah et al. refined this to the
odd- and even-restricted functions $\sigma_o\mathrm{mex}(n)$ and
$\sigma_e\mathrm{mex}(n)$. In this paper, we introduce and study the
overpartition analogue $\overline{\sigma_o\mathrm{mex}}(n)$, defined as
the sum of odd minimal excludants over all overpartitions of $n$. We
first derive the exact generating function for
$\overline{\sigma_o\mathrm{mex}}(n)$, and relate it to the bivariate
generating function of Aricheta and Donato for the overpartition
minimal excludant. Using elementary $q$-series arguments, together with
a weight-one eta-quotient identity for $\varphi(q)^2$ verified via the
Gordon--Hughes--Newman--Ligozat criterion, we establish an infinite
family of congruences satisfied by $\overline{\sigma_o\mathrm{mex}}(n)$.
Consequently, we obtain that $\overline{\sigma_o\mathrm{mex}}(0)=1$ and,
for $n\ge1$, $\overline{\sigma_o\mathrm{mex}}(n) \equiv 0 \pmod{4}$ if
and only if $n$ is a perfect square. We further obtain congruences for
the partial sums and self-convolution of $\overline{\sigma_o\mathrm{mex}}(n)$;
in particular, an infinite family of congruences modulo $8$ for the
self-convolution, expressed in terms of the divisor functions $d_1(n)$
and $d_3(n)$. We conclude the paper by establishing a Hardy--Ramanujan-type
asymptotic formula for $\overline{\sigma_o\mathrm{mex}}(n)$ via the
Wright circle method.
\end{abstract}



\section{Introduction}

A partition of $n\in \mathbb{N}$, is a finite non increasing seuence of positive integers
\[
\lambda_1 \geq \lambda_2 \geq \cdots \geq \lambda_m > 0,
\qquad \textnormal{such that } \qquad
\sum_{j=1}^{m} \lambda_j = n,
\]
where the $\lambda_j$'s are called the parts of the partition. Here and throughout this paper, we use the notation
\begin{align*}
 f_k := (q^k;q^k)_\infty= \prod_{n=1}^{\infty}(1-q^{nk}), \quad |q|<1.
\end{align*}

In 2004, Corteel and Lovejoy \cite{CL2004}, revisited the combinatorial object known as overpatitions. An overpartition of a positive integer $n$, is defined as a non-increasing sequence of positive integers summing to $n$ in which the first occurrence of each integer may be overlined. For example, the eight overpartitions
of $3$ are
\[
3,\ \overline{3},\ 2+1,\ \overline{2}+1,\
2+\overline{1},\ \overline{2}+\overline{1},\
1+1+1,\ \overline{1}+1+1.
\]
The number of overpartitions of $n$ is denoted by
$\overline{p}(n)$, and its generating function is given by
\begin{equation}
	\sum_{n=0}^{\infty} \overline{p}(n)\, q^n
	= \frac{f_2}{f_1^2}.
\end{equation}

For a set $S$ of positive integers, a minimal excludant of $S$ is the least positive integer that is not part of $S$. Andrews and Newman \cite{AN2019} introduced this in partition theory. They defined the minimal excludant of an integer partition $\lambda$, denoted by $\mathrm{mex}(\lambda)$ as the least positive integer that is not
part of $\lambda$.
With this, they also introduced an arithmetic function denoted by $\sigma\mathrm{mex}(n)$:
\begin{equation}
	\sigma\mathrm{mex}(n) :=
	\sum_{\lambda \in \mathcal{P}(n)}
	\mathrm{mex}(\lambda),
\end{equation}
where $\mathcal{P}(n)$ denotes the set of all partitions of $n$.
For example, the values for the minimal excludant for each partition of $n=5$ are: $\mathrm{mex}(5)=1$; $\mathrm{mex}(4+1)=2$; $\mathrm{mex}(3+2)=1$; $\mathrm{mex}(3+1+1)=2$; $\mathrm{mex}(2+2+1)=3$; $\mathrm{mex}(2+2+1)=3$; $\mathrm{mex}(1+1+1+1+1)=2$ with
$\sigma\mathrm{mex}(5) = 14$. If $p(n)$ denotes the number of partitions of $n$, they \cite{AN2019} also established the following identity:
\begin{equation}
	\sigma\mathrm{mex}(n) = p(n) +
	2\sum_{k=1}^{\infty} p(n - k^2).
\end{equation}

Recently, Baruah et al. \cite{BDR2023} explored the concept of minimal excludant functions and subsequently introduced the functions,
\begin{equation}
	\sigma_o\mathrm{mex}(n) :=
	\sum_{\lambda \in \mathcal{P}(n)}
	\mathrm{mex}_o(\lambda), \qquad
	\sigma_e\mathrm{mex}(n) :=
	\sum_{\lambda \in \mathcal{P}(n)}
	\mathrm{mex}_e(\lambda),
	\label{eq:BDR}
\end{equation}
where $\mathrm{mex}_o(\lambda)$
(resp.\ $\mathrm{mex}_e(\lambda)$) equals
$\mathrm{mex}(\lambda)$ if $\mathrm{mex}(\lambda)$ is
odd (resp.\ even) and equals $0$ otherwise.
Further, Barman and Singh \cite{BS2024} studied the arithmetic
properties and obtained asymptotic formulae for
$\sigma_o\mathrm{mex}(n)$ and
$\sigma_e\mathrm{mex}(n)$, proving lacunarity modulo
arbitrary powers of $2$ and established
\begin{equation}
	\sigma_o\mathrm{mex}(n) \sim
	\sigma_e\mathrm{mex}(n) \sim
	\frac{1}{8\sqrt[4]{6n^3}}
	\exp\!\left(\pi\sqrt{\frac{2n}{3}}\right)
	\quad \text{as } n \to \infty.
\end{equation}

The extension of the minimal excludant to overpartitions has been approached from more than one direction. Aricheta and Donato \cite{Aricheta2024} defined $\overline{\mathrm{mex}}(\pi)$ using only the non-overlined parts of an overpartition $\pi$, and established the bivariate generating function
\begin{equation}
	M(z,q):=\sum_{n=0}^{\infty}\sum_{m=1}^{\infty}
	p^{\overline{\mathrm{mex}}}(m,n)\,z^mq^n
	=\frac{(-q;q)_\infty}{(q;q)_\infty}
	\sum_{m=1}^{\infty} z^mq^{\binom m2}(1-q^m),
	\label{eq:aricheta}
\end{equation}
which, upon setting $z=1$ in $\partial M/\partial z$, recovers $\sigma\overline{\mathrm{mex}}(n)=D_3(n)$, the number of partitions of $n$ into distinct parts of three colors. More recently, Dhar, Mukhopadhyay, and Sarma \cite{Dhar2025} introduced four further overpartition mex variants and related them to $q$-series of Ramanujan: $\mathrm{omex}(\pi)$, the smallest positive integer missing from $\pi$ regardless of overline status; $\mathrm{omoex}(\pi)$, the smallest missing odd integer, searched among odd integers alone; and their overlined-part analogues $\widetilde{\mathrm{omex}}(\pi)$ and $\widetilde{\mathrm{omoex}}(\pi)$, defined under the ordering $1<\bar1<2<\bar2<\cdots$.

	Additionally, there are extensive works that delves deeper in to the areas of mex and mex related functions, we refer the readers to \cite{AN2020,BM2020,Barman2021,CR2021,HSS2022,HSY2022,KLW2021,KLW2023,KBEM2022,Ray2023,YangZhou2023}.

We now transition to the study of the overpartition analogue of $\sigma_o\mathrm{mex}(n)$. The minimal excludant we use, $\mathrm{mex}(\lambda)$, depends only on which values appear as parts of $\lambda$, regardless of whether they are overlined, precisely Dhar, Mukhopadhyay and Sarma's $\mathrm{omex}(\pi)$ \cite[Definition 1.1]{Dhar2025}. What distinguishes the present paper is the arithmetic function built from it: following the Baruah,Das and Ray convention \eqref{eq:BDR} rather than an odd-restricted search, we retain $\mathrm{mex}(\lambda)$ only when it happens to be odd, rather than restricting the search itself to odd integers as $\mathrm{omoex}(\pi)$ does. So the summatory function below is genuinely different from any of the four studied in \cite{Dhar2025}. We establish the following definition:

\begin{Definition}
	For $n \geq 0$, we define
	\begin{equation}
	\overline{\sigma_o\mathrm{mex}}(n):=
		\sum_{\lambda \in \overline{\mathcal{P}}(n)}
		\mathrm{mex}_o(\lambda),
	\end{equation}
	where $\overline{\mathcal{P}}(n)$ denotes the set of
	all overpartitions of $n$, and
	$\mathrm{mex}_o(\lambda)$ equals
	$\mathrm{mex}(\lambda)$ if $\mathrm{mex}(\lambda)$
	is odd, and equals $0$ otherwise.
\end{Definition}

It is straight forward to verify that the minimal excludant of an overpartition
depends only on which values appear as parts, regardless
of whether they are overlined. For instance, we tabulate below the values of $\overline{\sigma_o\mathrm{mex}}(n)$ for $0\leq n \leq 7$.
	\begin{table}[h]
	\centering
	\setcellgapes{4pt} 
	\makegapedcells
	\begin{tabular}{|c|c|c|c|c|c|c|c|c|}
		\hline
	$n$ & $0$ & $1$ & $2$ & $3$ & $4$ & $5$ & $6$ & $7$  \\
			\hline
	 $\overline{\sigma_o\mathrm{mex}}(n)$ & $1$ & $0$ & $2$ & $14$ & $16$ & $30$ & $34$ & $74$	\\
	 \hline
	\end{tabular}
\end{table}\\

Since, the function $\mathrm{mex}(\lambda)$ depends only on
which values appear as parts, regardless of
overlines, we have
\begin{equation}
\overline{\sigma_o\mathrm{mex}}(n) = \sum_{\substack{k \geq 1 \\
			k \text{ odd}}} k \cdot
	\#\bigl\{\lambda \in \overline{\mathcal{P}}(n)
	: \mathrm{mex}(\lambda) = k\bigr\}.
	\label{eq:omex_split}
\end{equation}
Here, we sum only over odd $k$ because
$\mathrm{mex}_o(\lambda) = 0$ whenever
$\mathrm{mex}(\lambda)$ is even, so even
values of $k$ contribute nothing to the sum.
Also, for an overpartition $\lambda$ has
$\mathrm{mex}(\lambda) = k$ if and only if
each of $1, 2, \ldots, k-1$ appears in
$\lambda$ and $k$ does not appear.
For a given value $m$, we note that, for $m$ absent and $m$ free the generating functions are $1$ and $\dfrac{1+q^m}{1-q^m}$, respectively, while if $m$ is forced to appear the generating function is $\dfrac{2q^m}{1-q^m}$.

The mechanism behind the generating function of Theorem \ref{thm:gf} below is similar to Aricheta and Donato's construction of $M(z,q)$ in \eqref{eq:aricheta}: both decompose the overpartition by which small values are forced to appear, missing, or free to occur. The two constructions differ in only one aspect -- whether a forced value must appear specifically as a non-overlined part, as in \cite{Aricheta2024}, or may appear in either color, as here, and this single distinction is what distinguishes the two theories arithmetically. In \cite{Aricheta2024}, a value $k$ forced to appear as a non-overlined part contributes $(1+q^k)\cdot\dfrac{q^k}{1-q^k}$ to $M(z,q)$, whereas in our setup a value $k$ forced merely to appear -- overlined or not -- contributes $\dfrac{2q^k}{1-q^k}$. Summing the forced contributions for $k=1,\ldots,2j$ yields the same triangular exponent $q^{\binom{2j+1}{2}}=q^{j(2j+1)}$ in both theories, but with coefficient $1$ in \cite{Aricheta2024} and $4^j$ here. This factor of $4^j$, appearing in Theorem \ref{thm:gf} below, is precisely why $\sigma\overline{\mathrm{mex}}(n)=D_3(n)$ exactly in \cite{Aricheta2024}, while $\overline{\sigma_o\mathrm{mex}}(n)$ instead satisfies only congruences modulo powers of $2$ (Corollary \ref{thm:mod4equiv}, Corollary \ref{thm:parity}, Theorem \ref{thm:main}, Theorem \ref{thm:selfconv}) rather than a closed combinatorial identity.

\indent The overarching goal in this paper is to focus on arithmetic properties of $\overline{\sigma_o\mathrm{mex}}(n)$ by relying on these congruences. In particular, Theorem \ref{thm:gf} provides the exact generating function for $\overline{\sigma_o\mathrm{mex}}(n)$ for integers $n\geq 0$.
	\begin{theorem}
		\label{thm:gf}
		For all integers $n \geq 0$, we have
		\begin{equation}
			\sum_{n=0}^{\infty}\overline{\sigma_o\mathrm{mex}}(n)\, q^n
			= \frac{1}{(q;q)_{\infty}}
			\sum_{j=0}^{\infty} (2j+1) \cdot 4^j \cdot
			q^{j(2j+1)}
			(1-q^{2j+1})(-q^{2j+2};q)_{\infty}.
		\end{equation}
	\end{theorem}
	\noindent As a consequence of Theorem \ref{thm:gf}, we deduce arithmetic congruences for $\overline{\sigma_o\mathrm{mex}}(n)$ modulo 2 and 4.

	\begin{corollary}
		\label{thm:mod4equiv}
		For all $n \geq 0$, we have
		\begin{equation}
		\overline{\sigma_o\mathrm{mex}}(n) \equiv \overline{p}_{\geq 2}(n) \pmod{4},
		\end{equation}
		where $\overline{p}_{\geq 2}(n)$ denotes the number
		of overpartitions of $n$ with all parts at least $2$.
	\end{corollary}
	\begin{corollary}
		\label{thm:parity}
		For all $n \geq 1$, we have
		\begin{equation}
			\overline{\sigma_o\mathrm{mex}}(n) \equiv 0 \pmod{2}.
		\end{equation}
	\end{corollary}
	\begin{theorem}
		\label{thm:main}
		We have $\overline{\sigma_o\mathrm{mex}}(0)=1$. Moreover, for all $n \geq 1$, we have
		\begin{equation}
			\overline{\sigma_o\mathrm{mex}}(n) \equiv 0 \pmod{4}
			\text{ if and only if }
			n \text{ is a perfect square.}
		\end{equation}
	\end{theorem}

	\begin{theorem}
		\label{thm:conv}
		Let $\varphi(q) = \displaystyle{\sum_{n=-\infty}^{\infty} q^{n^2}}$. For all $n \geq 1$, we have
		\begin{align}
			\sum_{k=0}^{n} \overline{\sigma_o\mathrm{mex}}(k) \cdot
			\overline{p}_{\geq 2}(n-k)
			&\equiv 0 \pmod{4},
			\label{eq:conv1} \\
			\sum_{k=0}^{n} \overline{\sigma_o\mathrm{mex}}(k) \cdot
			[q^{n-k}]\varphi(q)
			&\equiv 2 \pmod{4}.
			\label{eq:conv2}
		\end{align}
	\end{theorem}

	\begin{theorem}
		\label{thm:partial}
		For all $n \geq 0$, we have
		\begin{equation}
			\sum_{k=0}^{n} \overline{\sigma_o\mathrm{mex}}(k)
			\equiv 2\bigl(n - \lfloor\sqrt{n}\rfloor\bigr) + 1
			\pmod{4}.
		\end{equation}
	\end{theorem}

	\begin{theorem}
		\label{thm:selfconv}
		Let $d_1(n)$ (resp.\ $d_3(n)$) denote the number of
		divisors of $n$ congruent to $1$ (resp.\ $3$) modulo
		$4$. For all $n \geq 1$, we have
		\begin{equation}
			\sum_{k=0}^{n}\overline{\sigma_o\mathrm{mex}}(k) \cdot \overline{\sigma_o\mathrm{mex}}(n-k)
			\equiv 4\bigl(n + d_1(n) - d_3(n)\bigr) \pmod{8}.
			\label{eq:selfconv}
		\end{equation}
		In particular, the self-convolution sum is congruent
		to $0$ modulo $8$ if and only if $n$ is an odd perfect
		square, or $n = 2^a m$ where $a \geq 1$ and $m$ is an
		odd non-square integer.
	\end{theorem}
		We conclude this paper by establishing the following asymptotic formula for $\overline{\sigma_o\mathrm{mex}}(n)$.
	\begin{theorem}\label{thm:asymp}
As $n\to\infty$,
\[
\overline{\sigma_o\mathrm{mex}}(n)\ \sim\
\frac{\sqrt2}{16}\, n^{-3/4}\exp\bigl(\pi\sqrt n\bigr).
\]
\end{theorem}

	The remainder of this paper is organized as follows: In Section 2, we recall some preliminary definitions and results required in the proofs of our main theorems, which are presented in Sections 4 and 5. Section 3 is devoted to obtaining the generating function of $\overline{\sigma_o\mathrm{mex}}(n)$, while the proof of Theorem \ref{thm:asymp} is given in Section 6.

\section{Preliminaries}
\label{sec:prelim}

In this section, we present the necessary $q$-series definitions and results, together with the modular-forms background required for Corollary~\ref{cor:phisq} in Section~\ref{sec:infinite}. For a detailed account of these results, we refer the reader to \cite{Berndt2006,Ono2004}.

\indent We first recall Ramanujan's general theta function $f(a,b)$ \cite[p.~35]{Berndt1991}, defined by
\begin{equation}
	f(a,b):=\sum_{n=-\infty}^{\infty}
	a^{\frac{n(n+1)}{2}}
	b^{\frac{n(n-1)}{2}},
	\qquad |ab|<1.
\end{equation}
The product representation of $f(a,b)$ follows from the Jacobi triple product identity \cite[p.~35]{Berndt1991}:
\begin{equation}
	f(a,b)=(-a,ab)_\infty\,(-b,ab)_\infty\,(ab,ab)_\infty.
\end{equation}
Two special cases of Ramanujan's theta functions are defined by
\begin{align}
	\varphi(q) &:=
	\sum_{n=-\infty}^{\infty} q^{n^2}
	= \frac{f_2^5}{f_1^2 f_4^2},
	\label{eq:phi} \\
	\varphi(-q) &:=
	\sum_{n=-\infty}^{\infty} (-1)^n q^{n^2}
	= \frac{f_1^2}{f_2}.
	\label{eq:phiminus}
\end{align}

\begin{Definition}
The Klein four-group is a finite abelian group with four elements in which each element is its inverse. Klein's four-group is denoted by $K_4$, in fact $K_4$ is the smallest non-cyclic group.
\end{Definition}
For example, $S=\{1,3,5,7\}$ under multiplication modulo $8$ forms Klein's four-group.\\
\indent For a positive integer $N$, $\Gamma_0(N)$ denotes the subgroup of $SL_2(\mathbb{Z})$ defined by
\begin{align*}
	\Gamma_0(N):=
	\left\{
	\begin{bmatrix}
		a & b \\
		c & d
	\end{bmatrix} \in SL_2(\mathbb{Z})
	\,:\, c \equiv 0 \pmod N
	\right\}.
\end{align*}
$\Gamma_0(N)$ acts on $\mathbb{H}=\{z: \operatorname{Im}(z)>0\}$ by
$\gamma z:=\dfrac{az+b}{cz+d}$ for $\gamma=\begin{bmatrix}a&b\\c&d\end{bmatrix}\in\Gamma_0(N)$.
If $\chi$ is a Dirichlet character modulo $N$ and $k$ is a positive integer, a meromorphic function $f(z)$ on $\mathbb{H}$ satisfying
$f(\gamma z)=\chi(d)(cz+d)^kf(z)$ for all $\gamma\in\Gamma_0(N)$, $z\in\mathbb{H}$, is a modular form of weight $k$ and Nebentypus character $\chi$ with respect to $\Gamma_0(N)$; if $f(z)$ is holomorphic on $\mathbb{H}$ and at every cusp of $\Gamma_0(N)$, it is a \emph{holomorphic} modular form, and the space of such forms is denoted $M_k(\Gamma_0(N),\chi)$.\\
\indent The Dedekind's eta-function $\eta(z)$ defined by
\begin{align}\label{eq:eta-def}
	\eta(z):=q^{1/24}(q;q)_\infty=q^{1/24}\prod_{n=0}^\infty(1-q^{n}), \text{ where } q=e^{2\pi iz},
\end{align}
which is a non-vanishing holomorphic function on $\mathbb{H}=\{z: \operatorname{Im}(z)>0\}$.\\
Further, $\eta(z)$ satisfies the following modular transformations \cite[Theorem~1.61]{Ono2004}
\begin{align}
\eta(z+1)&=e^{\pi i/12}\eta(z)\label{tr1}\\ \eta\!\left(-\frac1z\right)&=\sqrt{-iz}\,\eta(z)\label{tr2}.
\end{align}
A function is called an eta-quotient if it is of the form
\begin{align}
	f(z)=\prod_{\delta\mid N}\eta(\delta z)^{r_\delta},
\end{align}
where $N$ is a positive integer and $r_\delta$ is an integer. The following criterion identifies when an eta-quotient is a holomorphic modular form.
\begin{theorem}[{{\cite[Theorem~1.64 and 1.65]{Ono2004}}}]\label{t 2.1}
	If $f(z)=\prod_{\delta\mid N}\eta(\delta z)^{r_\delta}$ is an eta-quotient with
	$k=\frac12\sum_{\delta\mid N}r_\delta\in\mathbb{Z}$, and satisfies
	\begin{align}\label{2.3}
		\sum_{\delta\mid N}\delta\, r_\delta\equiv0\pmod{24},
	\end{align}
	\begin{align}\label{2.4}
		\sum_{\delta\mid N}\frac{N}{\delta}\,r_\delta\equiv0\pmod{24},
	\end{align}
	\begin{align}\label{2.5}
		\sum_{\delta\mid N}\frac{\gcd(d,\delta)^2 r_\delta}{\delta}\ge0,\qquad\text{for every } d\mid N,
	\end{align}
	then $f(z)\in M_k(\Gamma_0(N),\chi)$, where
	$\chi(d):=\left(\dfrac{(-1)^k\prod_{\delta\mid N}\delta^{r_\delta}}{d}\right)$.
\end{theorem}

\section{Generating Function for $	\overline{\sigma_o\mathrm{mex}}(n)$ }
\label{sec:gf}

In this section, we establish the generating function for $	\overline{\sigma_o\mathrm{mex}}(n)$.

\begin{proof}[Proof of Theorem \ref{thm:gf}]
Let $F_{2j+1}(q)$ denote the generating function for $	\overline{\sigma_o\mathrm{mex}}(n)$, we have
	\begin{align*}
		F_{2j+1}(q) &=
		\underbrace{\prod_{m=1}^{2j}
			\frac{2q^m}{1-q^m}}_{\text{values }
			1,\ldots,2j \text{ forced}}
		\cdot\;
		\underbrace{1}_{\text{value }
			2j+1 \text{ absent}}
		\cdot\;
		\underbrace{\prod_{m=2j+2}^{\infty}
			\frac{1+q^m}{1-q^m}}_{\text{values }
			> 2j+1 \text{ free}}.
		\\
	& =\frac{4^j \cdot q^{j(2j+1)}}{(q;q)_{2j}} \cdot \frac{(-q^{2j+2};q)_\infty}
	{(q^{2j+2};q)_\infty} \\
	& =\frac{4^j \cdot q^{j(2j+1)}}{(q;q)_{2j}} \cdot  (-q^{2j+2};q)_\infty \cdot
	\frac{(q;q)_{2j}(1-q^{2j+1})}
	{(q;q)_\infty}\\
	&=	\frac{4^j q^{j(2j+1)}
		(1-q^{2j+1})(-q^{2j+2};q)_\infty}
	{(q;q)_\infty}.
	\end{align*}
Therefore, we obtain
	\begin{align*}
		\sum_{n=0}^{\infty}	\overline{\sigma_o\mathrm{mex}}(n)q^n
		&= \sum_{j=0}^{\infty}
		(2j+1)\cdot F_{2j+1}(q) \\[4pt]
		&= \frac{1}{(q;q)_\infty}
		\sum_{j=0}^{\infty}(2j+1)\cdot 4^j
		\cdot q^{j(2j+1)}(1-q^{2j+1})
		(-q^{2j+2};q)_\infty,
	\end{align*}
	since $	\overline{\sigma_o\mathrm{mex}}(n)$ weights each overpartition
by its odd mex value, this completes the proof.
\end{proof}
\begin{remark}
For $n = 3$, the eight overpartitions
of $3$ with their mex values are shown in the table below. Clearly, $	\overline{\sigma_o\mathrm{mex}}(3) = 2(1) + 4(3) + 2(0) = 14$.
\begin{table}[h]
	\centering
	\setcellgapes{4pt} 
	\makegapedcells
	\begin{tabular}{|c|c|c|c|}
		\hline
		Overpartition & Parts present & $\mathrm{mex}$ & $\mathrm{mex}_0$\\
		\hline
		$3$, $\bar{3}$ & $\{3\}$ & $1$ & $1$\\
		\hline
		$2+1$, $\bar{2}+1$, $2+\bar{1}$, $\bar{2}+\bar{1}$ & $\{1,2\}$ & $3$ & $3$\\
		\hline
		$1+1+1$, $\bar{1}+1+1$ & $\{1\}$ & $2$ & $0$\\
		\hline
	\end{tabular}
\end{table}
\end{remark}

\section{Proof of Corollary~\ref{thm:mod4equiv}, Corollary~\ref{thm:parity} and Theorem~\ref{thm:main}}
\label{sec:main}
We first prove the following lemmas which provides an essential argument to prove our results.

\begin{lemma}
	\label{lem:mod4}
	For all $n\ge0$, we have
	\begin{equation}
		\sum_{n=0}^{\infty} 	\overline{\sigma_o\mathrm{mex}}(n)\, q^n
		\equiv \frac{(-q^2;q)_\infty}{(q^2;q)_\infty}
		=: G(q) \pmod{4}.
		\label{eq:mod4reduce}
	\end{equation}
\end{lemma}

\begin{proof}
Thanks to Theorem \ref{thm:gf} and the congruence, $4^j \equiv 0 \pmod{4}$, for $j \geq 1$, we complete the proof of Lemma \ref{lem:mod4}.
\end{proof}

\begin{lemma}
	\label{lem:keyid}
	We have
	\begin{equation}
		G(q) \equiv \frac{2}{1-q} - \varphi(q) \pmod{4}.
		\label{eq:keyid}
	\end{equation}
\end{lemma}
\begin{proof}
For $n\ge2$, we have the exact identity
\[
\frac{1+q^n}{1-q^n}=1+\frac{2q^n}{1-q^n}=1+2\sum_{k\ge1}q^{nk}.
\]
Multiplying these factors together for $n\ge2$ and reducing modulo $4$ -- any cross-term between the ``$2(\cdot)$'' parts of two distinct factors contributes a multiple of $4$, and at each fixed power of $q$ this is a finite computation -- we obtain
\[
G(q)\equiv 1+2\sum_{n\ge2}\sum_{k\ge1}q^{nk}\pmod4.
\]
For fixed $N\ge1$, the double sum counts pairs $(n,k)$ with $n\ge2$, $k\ge1$, $nk=N$, i.e.\ divisors $n\ge2$ of $N$; there are $d(N)-1$ of these, where $d(N)$ denotes the number of positive divisors of $N$. Hence
\[
G(q)\equiv1+2\sum_{N\ge1}\bigl(d(N)-1\bigr)q^N
=1-\frac{2q}{1-q}+2D(q)\pmod4,\qquad D(q):=\sum_{N\ge1}d(N)q^N.
\]
Since $1-\dfrac{2q}{1-q}=3-\dfrac{2}{1-q}$, this gives
\[
G(q)\equiv3-\frac2{1-q}+2D(q)\pmod4,
\]
so the desired congruence $G(q)\equiv\dfrac2{1-q}-\varphi(q)\pmod4$ is equivalent to
\[
2D(q)+\varphi(q)\equiv1\pmod4,
\]
using that $4/(1-q)\equiv0\pmod4$. Since $\varphi(q)=1+2\sum_{n\ge1}q^{n^2}$, this reduces further to
\[
D(q)\equiv\sum_{n\ge1}q^{n^2}\pmod2.
\]
This is a classical fact: the divisors of $N$ pair up as $n\leftrightarrow N/n$, an involution on the divisor set of $N$ with a fixed point exactly when $n=N/n=\sqrt N$. Hence $d(N)$ is odd if and only if $N$ is a perfect square, which is precisely the coefficient of $q^N$ in $\sum_{n\ge1}q^{n^2}$. This proves $D(q)\equiv\sum_{n\ge1}q^{n^2}\pmod2$, and combining the above completes the proof.
\end{proof}
 We now provide the proof of Corollary \ref{thm:mod4equiv}.
\begin{proof}[Proof of Corollary~\ref{thm:mod4equiv}]
	By Lemma~\ref{lem:mod4}, we have
	\[
	\sum_{n=0}^{\infty}	\overline{\sigma_o\mathrm{mex}}(n)q^n
	\equiv\prod_{m=2}^{\infty}\frac{1+q^m}{1-q^m}\pmod{4}=  \sum_{n=0}^{\infty}
	\overline{p}_{\geq 2}(n)\, q^n,
	\]
	where $\overline{p}_{\geq 2}(n)$ denotes the
	number of overpartitions of $n$ with all
	parts at least $2$. Extracting the coefficients of $q^n$ on both the sides, we complete the proof.
\end{proof}

\begin{proof}[Proof of Corollary~\ref{thm:parity}]
	Thanks to Corollary~\ref{thm:mod4equiv}, we have
	\[
	\overline{\sigma_o\mathrm{mex}}(n) \equiv \overline{p}_{\geq 2}(n)
	\pmod{2}.
	\]
	To complete the proof of Corollary \ref{thm:parity}, it is enough show for all $n \geq 1$,
	$\overline{p}_{\geq 2}(n) \equiv 0 \pmod{2}$, which immediately follows from the fact, $(1 + q^m) \equiv (1 - q^m) \pmod{2}$, for each $m\ge2$.
\end{proof}

\begin{proof}[Proof of Theorem~\ref{thm:main}]
	From Lemma~\ref{lem:mod4} and Lemma~\ref{lem:keyid}, we have
	\[
		\sum_{n=0}^{\infty}\overline{\sigma_o\mathrm{mex}}(n)q^n \equiv\frac{2}{1-q} - \varphi(q)
		\pmod{4}.
	\]
	Since $\varphi(q)=1+2\sum_{n\ge1}q^{n^2}$, the coefficient of $q^0$ on the right side is $2-1=1$, which matches $\overline{\sigma_o\mathrm{mex}}(0)=1$.
	For $n\ge1$, extracting the coefficient of $q^n$ from both sides of the above congruence, we obtain
	\[
	\overline{\sigma_o\mathrm{mex}}(n)\equiv 2-
	\begin{cases}
		2 & \text{if } n \text{ is a
			perfect square}, \\
		0 & \text{otherwise},
	\end{cases}
\pmod{4}	\]
	This completes the proof of Theorem \ref{thm:main}.
\end{proof}

\section{Proof of Theorem \ref{thm:conv}-\ref{thm:selfconv} }
\label{sec:infinite}
We first prove the following lemmas. Furthermore, the structural observation in Lemma \ref{lk4} is the key to establish the theorems in this section. We begin by identifying $\varphi(q)^2$ as a weight-one eta-quotient, which puts Jacobi's two-square theorem on a modular-forms footing.

\begin{corollary}\label{cor:phisq}
The eta-quotient $\varphi(q)^2=\dfrac{f_2^{10}}{f_1^4f_4^4}$ is a holomorphic
modular form of weight $1$ on $\Gamma_0(4)$ with Nebentypus character
$\chi_{-4}$; that is, $\varphi(q)^2\in M_1(\Gamma_0(4),\chi_{-4})$.
Consequently,
\[
\varphi(q)^2=1+4\sum_{n=1}^{\infty}\bigl(d_1(n)-d_3(n)\bigr)q^n,
\]
where $d_1(n)$ and $d_3(n)$ denote the number of divisors of $n$ that are
$\equiv1$ and $\equiv3\pmod4$ respectively.
\end{corollary}

\begin{proof}
Write $\varphi(q)^2=\eta(z)^{-4}\eta(2z)^{10}\eta(4z)^{-4}$, an eta-quotient
of level $N=4$ with exponents $r_1=-4$, $r_2=10$, $r_4=-4$. We verify the
hypotheses of Theorem~\ref{t 2.1}.

\emph{Weight.} $k=\tfrac12\sum_{\delta\mid4}r_\delta
=\tfrac12(-4+10-4)=1\in\mathbb Z_{>0}$.

\emph{Condition \eqref{2.3}.}
$\sum_{\delta\mid4}\delta r_\delta
=1(-4)+2(10)+4(-4)=-4+20-16=0\equiv0\pmod{24}$.

\emph{Condition \eqref{2.4}.}
$\sum_{\delta\mid4}\dfrac{4}{\delta}r_\delta
=4(-4)+2(10)+1(-4)=-16+20-4=0\equiv0\pmod{24}$.

\emph{Condition \eqref{2.5}}, checked for each $d\mid4$:
\begin{align*}
d=1:&\quad \frac{1^2(-4)}{1}+\frac{1^2(10)}{2}+\frac{1^2(-4)}{4}=-4+5-1=0\ge0,\\
d=2:&\quad \frac{1^2(-4)}{1}+\frac{2^2(10)}{2}+\frac{2^2(-4)}{4}=-4+20-4=12\ge0,\\
d=4:&\quad \frac{1^2(-4)}{1}+\frac{2^2(10)}{2}+\frac{4^2(-4)}{4}=-4+20-16=0\ge0.
\end{align*}
All three conditions hold with $k=1\in\mathbb{Z}$, so by
Theorem~\ref{t 2.1}, $\varphi(q)^2\in M_1(\Gamma_0(4),\chi)$, where
\[
\chi(d)=\left(\frac{(-1)^k\prod_{\delta\mid4}\delta^{r_\delta}}{d}\right)
=\left(\frac{-1\cdot 2^{10}\cdot4^{-4}}{d}\right)
=\left(\frac{-4}{d}\right)=\chi_{-4}(d).
\]
The stated $q$-expansion is Jacobi's classical two-square identity
\cite{Berndt2006}, which is now identified as the Fourier expansion of an
explicit weight-$1$ Eisenstein series attached to $\chi_{-4}$.
\end{proof}

\begin{lemma}
	\label{lem:jacobi}
	We have
	\begin{equation}
		\varphi(q)^2 \equiv 1 \pmod{4}.
	\end{equation}
\end{lemma}

\begin{proof}
By Corollary~\ref{cor:phisq},
	\[
\varphi(q)^2 = 1 + 4\sum_{n=1}^{\infty}
(d_1(n) - d_3(n))q^n,
\]
and the proof follows immediately.
\end{proof}

\begin{lemma}\label{lk4}
Define $G(q) =
\dfrac{(-q^2;q)_{\infty}}{(q^2;q)_{\infty}}$, then the set $K'=\bigl\{1, G(q), \varphi(q),
G(q)\varphi(q)\bigr\}$ forms Klein's four-group under multiplication modulo $4$.
\end{lemma}
\begin{proof}
	We have
	\begin{equation}
		G(q)^2 \equiv 1, \quad
		\varphi(q)^2 \equiv 1, \quad
		G(q)\varphi(q) \equiv \frac{1+q}{1-q} \pmod{4},
		\label{eq:klein}
	\end{equation}
which can be easily verified using \eqref{eq:keyid} and Lemma \ref{lem:jacobi}. Hence, the set $K'$ forms a Klein's four-group. This completes the proof of Lemma \ref{lk4}.
\end{proof}

\begin{proof}[Proof of Theorem~\ref{thm:conv}]
	Using Lemma~\ref{lem:mod4} and \eqref{eq:klein}, we have
	\begin{align}\label{e1}
		\left(\sum_{n \geq 0}	\overline{\sigma_o\mathrm{mex}}(n)q^n\right)
	\cdot G(q) \equiv 1 \pmod{4}
	\end{align}
and
\begin{align}\label{e2}
\left(\sum_{n \geq 0}	\overline{\sigma_o\mathrm{mex}}(n)q^n\right)
\cdot \varphi(q)
\equiv \frac{1+q}{1-q} \pmod{4}.
\end{align}
More precisely, \eqref{eq:conv1} and \eqref{eq:conv2} easily follows from \eqref{e1} and \eqref{e2}, respectively.
\end{proof}

\begin{proof}[Proof of Theorem~\ref{thm:partial}]
	Multiplying $\sum_{k=0}^n	\overline{\sigma_o\mathrm{mex}}(n)q^n \equiv G(q) \pmod{4}$
	by $\frac{1}{1-q}$ and employing Lemma~\ref{lem:keyid}, we obtain
	\[
	\sum_{n \geq 0}
	\left(\sum_{k=0}^n	\overline{\sigma_o\mathrm{mex}}(k)\right)q^n
	\equiv \frac{2}{(1-q)^2} -
	\frac{\varphi(q)}{1-q} \pmod{4}.
	\]
	Extracting the coefficients of $q^n$ from the above congruence, we complete the proof of Theorem \ref{thm:partial}.
\end{proof}

\begin{proof}[Proof of Theorem~\ref{thm:selfconv}]
	Thanks to Lemma~\ref{lem:mod4}, we have
	\[
	\left(\sum_{n \geq 0}	\overline{\sigma_o\mathrm{mex}}(n)q^n\right)^2
	\equiv G(q)^2 \pmod{8}.
	\]
On using \eqref{eq:keyid} in above congruence, we obtain
	\begin{align*}
		G(q)^2 &\equiv
		\left(\frac{2}{1-q} - \varphi(q)\right)^2 \\
		&= \frac{4}{(1-q)^2} -
		\frac{4\varphi(q)}{1-q} +
		\varphi(q)^2 \pmod{8}.
	\end{align*}
Further, on employing Corollary~\ref{cor:phisq}, we have
	\[
	\varphi(q)^2 = 1 + 4\sum_{n \geq 1}
	(d_1(n)-d_3(n))q^n.
	\]
	Extracting the coefficients of $q^n$, for $n \geq 1$,
	\begin{align*}
		[q^n]G^2 &\equiv 4(n+1) -
		4(1+2\lfloor\sqrt{n}\rfloor) +
		4(d_1(n)-d_3(n)) \pmod{8} \\
		&= 4n - 8\lfloor\sqrt{n}\rfloor +
		4(d_1(n)-d_3(n)) \\
		&\equiv 4(n + d_1(n) - d_3(n)) \pmod{8},
	\end{align*}
which is \eqref{eq:selfconv}.

	For the characterization: writing $n = 2^a m$
	with $m$ odd, we have
	$d_1(n)-d_3(n) \equiv d(m) \pmod{2}$,
	where $d(m)$ is the number of divisors of $m$.
	Therefore, the sum is congruent to $0$ modulo $8$ if and only if
	$n \equiv d(m) \pmod{2}$, which holds if and only if
	$a \geq 1$ with $m$ a non-square, or
	$a = 0$ with $n$ an odd perfect square. This completes the proof of Theorem \ref{thm:selfconv}.
\end{proof}

\section{Asymptotic Formula for 	$\overline{\sigma_o\mathrm{mex}}(n)$}
\label{sec:asymp}
In this section, we prove an asymptotic formula for $\overline{\sigma_o\mathrm{mex}}(n)$.

\begin{lemma}\label{lem:trunc-theta}
Let $K:=K(t)\to\infty$ as $t\to 0^{+}$ with $Kt\to 0$ and $K^{2}t=O(1)$.
Then
\[
\ln\prod_{k=1}^{K}\bigl(1+e^{-kt}\bigr)=K\ln 2-\frac{K^{2}t}{4}+o(1),
\qquad t\to 0^{+}.
\]
\end{lemma}

\begin{proof}
For $s\ge0$, we observe $1+e^s=e^s(1+e^{-s})$, which implies $\gamma(-s)=s+\gamma(s)$, where $\gamma(s)=\ln(1+e^{-s})$.\\ We now consider
\[
\gamma(s)=\ln2-\frac s2+\frac{s^2}8+O(s^4),
\]
which implies
\[
\sum_{k=1}^{K}\gamma(kt)=K\ln2-\frac t2\sum_{k=1}^{K}k+\frac{t^2}8\sum_{k=1}^{K}k^2
+O\Bigl(t^4\sum_{k=1}^{K}k^4\Bigr).
\]
Since $Kt=o(1)$ and $K^2t=O(1)$, Lemma \ref{lem:trunc-theta} follows.
\end{proof}

\begin{proof}[Proof of Theorem~\ref{thm:asymp}]
We set $2j+1=K$ in Theorem \ref{thm:gf}, to obtain
\begin{align*}
F(e^{-t})=\frac1{f_1(t)}\sum_{K\text{ odd}}T_K(t),
\end{align*}
where
\begin{align}
T_K(t):=K\cdot2^{K-1}e^{-\frac{K^2-K}2t}\bigl(1-e^{-Kt}\bigr)\!\!\prod_{k\ge K+1}\!\!\bigl(1+e^{-kt}\bigr).\label{e4}
\end{align}
For $q=e^{-t}$, we set $z=it/2\pi$ and $z=it/\pi$ in \eqref{tr1} and \eqref{tr2}, respectively, we obtain as $t\to0^+$ the following asymptotic formulas
\begin{align}
	f_1(t)&:=(q;q)_\infty\sim\sqrt{\frac{2\pi}t}\,e^{-\pi^2/6t}\label{f1}\\
	f_2(t)&:=(q^2;q^2)_\infty\sim\sqrt{\frac\pi t}\,e^{-\pi^2/12t}.\label{f2}
\end{align}
Now, fix $u\ge0$ and let $K$ be the odd integer nearest $1+ut^{-1/2}$, such that
$K=ut^{-1/2}+O(1)$. Since
\[
\prod_{k\ge K+1}\bigl(1+e^{-kt}\bigr)=\frac{f_2(t)/f_1(t)}{\prod_{k=1}^K\bigl(1+e^{-kt}\bigr)},
\]
from Lemma~\ref{lem:trunc-theta}, we have
\begin{align}
2^{K-1}\prod_{k\ge K+1}\bigl(1+e^{-kt}\bigr)
&=\frac{2^K}2\cdot\frac{f_2(t)/f_1(t)}{\exp\bigl(K\ln2-\tfrac{u^2}4+o(1)\bigr)}\nonumber\\
&=\frac12e^{u^2/4}\,\frac{f_2(t)}{f_1(t)}\,\bigl(1+o(1)\bigr).\label{e3}
\end{align}
From identities \eqref{e4} and \eqref{e3}, we have
\[
T_K(t)=\frac{u^2}2\,e^{-u^2/4}\,\frac{f_2(t)}{f_1(t)}\,\bigl(1+o(1)\bigr).
\]
Since $T_K(t)\le K2^Kq^{(K^2-K)/2}$ is exponentially small, the Riemann sum can be extended to an integral over $[0,\infty)$. Thus, we have
\begin{align*}
\sum_{K\text{ odd}}T_K(t)
&=\frac1{2\sqrt t}\int_0^\infty\frac{u^2}2e^{-u^2/4}\,du\cdot\frac{f_2(t)}{f_1(t)}\,(1+o(1))\\
&=\frac{\sqrt\pi}{2\sqrt t}\,\frac{f_2(t)}{f_1(t)}\,(1+o(1)).
\end{align*}
Also, from \eqref{f1} and \eqref{f2}, we have
\begin{align*}
\dfrac{f_2(t)}{f_1(t)}\sim\tfrac1{\sqrt2}e^{\pi^2/12t}.
\end{align*}
Therefore, we conclude
\[
F(e^{-t})\ \sim\ \frac{\sqrt{\pi/8t}}{\sqrt{2\pi/t}}
\exp\Bigl(\frac{\pi^2}{12t}+\frac{\pi^2}{6t}\Bigr)
=\frac14\exp\Bigl(\frac{\pi^2}{4t}\Bigr),\qquad t\to0^+.
\]
Finally, from Cauchy's theorem, we have
\begin{align*}
	\overline{\sigma_o\mathrm{mex}}(n)=\tfrac1{2\pi}\int_{-\pi}^\pi
	F(re^{i\theta})r^{-n}e^{-in\theta}\,d\theta,
\end{align*}
where $r=e^{-t}$. The Wright's
circle method \cite{Wright1969}, justifies the shift of
real-axis asymptotic to the coefficient asymptotics. Applying the
saddle-point method with exponent $\psi(t)=\pi^2/4t+nt$ minimized at $t_0=\pi/(2\sqrt n)$, we have
\[
\psi(t_0)=\pi\sqrt n,\qquad \psi''(t_0)=\frac{4n^{3/2}}\pi,\qquad
\sqrt{2\pi\,\psi''(t_0)}=2\sqrt2\,n^{3/4}.
\]
Therefore, saddle-point approximation gives
\[
\overline{\sigma_o\mathrm{mex}}(n)\ \sim\ \frac{1/4}{2\sqrt2\,n^{3/4}}\,e^{\pi\sqrt n}
=\frac{\sqrt2}{16}\,n^{-3/4}\exp(\pi\sqrt n),
\]
which completes the proof.
\end{proof}
\section*{Funding}
The authors declare that no funds, grants, or other support were received during the preparation of this manuscript.


\end{document}